\documentclass{birkjour}
 \newtheorem{thm}{Theorem}[section]

 \theoremstyle{definition}
 \newtheorem{defn}[thm]{Definition}
 \theoremstyle{remark}
 \newtheorem{rem}[thm]{Remark}
 \newtheorem*{ex}{Example}
 \numberwithin{equation}{section}

\begin{document}

%
%
%
%
%
%
%
%
%

\title[Some results on harmonic metrics]
 {Some results on harmonic metrics}

\author[Fatima Zohra Kadi]{Fatima Zohra Kadi}

\address{%
Laboratory of Quantum Physics and Mathematical Modeling (LPQ3M) \\
University of Mascara\\
Faculty of Exact Sciences\\
Department of Mathematics \\
BP 305 Route de Mamounia\\
29000 Mascara\\
Algeria}
\email{fatima.kadi@univ-mascara.dz}

\author{Bouazza Kacimi}
\address{%
University of Mascara\\
Faculty of Exact Sciences\\
Department of Mathematics \\
BP 305 Route de Mamounia\\
29000 Mascara\\
Algeria}
\email{bouazza.kacimi@univ-mascara.dz}

\author{Mustafa \"{O}zkan}
\address{ Gazi University \\
Faculty of Sciences\\
Department of Mathematics \\
06500 Ankara\\
Turkey}
\email{ozkanm@gazi.edu.tr}



\date{January 1, 2004}

\begin{abstract}
In this paper, we prove that the harmonicity of a metric on a pseudo-Riemannian manifold is equivalent to the harmonicity of both its Sasaki (resp. horizontal and complete) lift metric on the tangent bundle and its Sasaki lift metric on the cotangent bundle. Some applications are included.
\end{abstract}
\maketitle
\section{Introduction}
The notion of harmonic maps between (pseudo-)Riemannian manifolds grew out of essential notions in differential geometry, such as geodesics, minimal surfaces and harmonic functions (i.e., real $C^{2}$ functions belonging to the kernel of the Laplace operator). Harmonic maps are solutions to a natural geometrical variational problem, this means that they satisfy the Euler-Lagrange systems see \cite{eells1,eells2}, they are closely related to analysis, to the general relativity, and to nonlinear field theory in theoretical physics see \cite{gauduchon,xin}. This subject has been developed extensively see for example \cite{baird,eells2,eells3,xin} and references therein. In \cite{chen2}, B. Y. Chen and T. Nagano introduced the notions of harmonic metrics and harmonic tensors and obtain a necessary and sufficient condition for an immersion between two Riemannian manifolds
to be harmonic, for recent developments of this topic see \cite{chen1}. C. L. Bejan and S. L. Dru\c{t}\u{a}-Romaniuc in \cite{bejan} proved that the harmonicity of a metric on four-dimensional Walker manifold of signature (2,2) and that of its Sasaki (resp. horizontal) lift metric on the tangent bundle are equivalent, the same conclusion is obtained by A. Zaeim and P. Atashpeykarin in \cite{zaeim1} and  by  A. Zaeim, M. Jafari and M. Yaghoubi in \cite{zaeim2} on non-reductive homogeneous manifold of dimension four and on G\"{o}del-type space time respectively but, in both cases they proved also the equivalence for the complete lift metric on the tangent bundle.

The aim of this note is twofold. The first is to study the relationship between the harmonicity of a metric on a pseudo-Riemannian manifold and that of its Sasaki (resp. horizontal and complete) lift metric on the tangent bundle in more general case. The second is to highlight the relationship between the harmonicity of a metric on a pseudo-Riemannian manifold and that of its Sasaki lift metric on the cotangent bundle. This paper is structured as follows. In section $2$, we give some definitions, properties and results which will be used in the sequel. In
section $3$, we study harmonic metrics lifted to the tangent bundle of a pseudo-Riemannian manifold and we prove that there is an equivalence between the harmonicity of a metric on a pseudo-Riemannian manifold and that of its Sasaki (resp. horizontal and complete) lift metric on the tangent bundle (see Theorem \ref{thm1}). As consequence, we recover the results of \cite{bejan,zaeim1} and \cite{zaeim2}. In section $4$, we deal with harmonic metrics on Egorov spaces (see Theorem \ref{thm2}) and give an example (see Example \ref{ex1}). The section $5$ is devoted to the study of harmonic metrics lifted to the cotangent bundle of a pseudo-Riemannian manifold, we prove also that the harmonicity of a metric and that of its Sasaki lift metric on the cotangent bundle are equivalent (see Theorem \ref{thm3}). As consequence, we obtain the main results of this note (see Theorem \ref{thm4}). In the last section, we apply Theorem \ref{thm3} on four-dimensional Walker manifold of signature (2,2), on G\"{o}del-type spacetime, on non-reductive homogeneous manifold of dimension four and on Egorov space.
\section{Preliminaries}
\subsection{Harmonic maps and harmonic metrics}
Let $(M,g)$ and $(N,h)$ be two pseudo-Riemannian manifolds of dimensions $m$ and $n$ respectively, and let $\varphi:(M,g)\rightarrow(N,h)$ a smooth map between them. The map $\varphi$ is said to be harmonic if it is a critical point of the energy functional defined by
\begin{equation*} E(\varphi,D)=\frac{1}{2}\underset{D}{\int}|d\varphi|^2v_g,
\end{equation*}
for any compact domain $D$ of $M$, where $v_g$ is the volume element of $(M,g)$. If $M$ is compact, we write
\begin{equation*} E(\varphi)=E(\varphi,M).
\end{equation*}
Let $\nabla d\varphi$ be the second fundamental form of $\varphi.$ The trace of $\nabla d\varphi$ with respect to $g$ is called the tension field of $\varphi$, and is denoted by $\tau(\varphi)$. The Euler-Lagrange equation of $E(\varphi,D)$ is \cite{eells3}
\begin{equation*}
\tau(\varphi)=tr(\nabla d\varphi)=0.
\end{equation*}
Let $U\subset M$ be an open set with coordinates $(x^1,\ldots,x^m)$ and $V \subset N$ be an open set with coordinates $(y^1,\ldots,y^n)$, such that $\varphi(U)\subset V$ and suppose that $\varphi$ is locally represented by
\begin{equation*}
y^\alpha=\varphi^\alpha(x^1,\ldots,x^m),\;\; \alpha=1,\ldots,n.
\end{equation*}
Then we have:
\begin{equation}
(\nabla d\varphi)^\gamma_{ij}=\frac{\partial^2\varphi^\gamma}{\partial x^i\partial x^j}- \Gamma^k_{ij}\frac{\partial \varphi^\gamma}{\partial x^k}+{}^N\Gamma^\gamma_{\alpha\beta}(\varphi)\frac{\partial \varphi^\alpha}{\partial x^i}\frac{\partial \varphi^\beta}{\partial x^j},\label{eq1}
\end{equation}
here $\Gamma^k_{ij}$ and $^N\Gamma^\gamma_{\alpha\beta}$ denote the Christoffel symbols of $(M,g)$ and $(N,h)$ respectively. So, $\varphi$ is harmonic if and only if
\begin{equation}\tau(\varphi)=tr (\nabla d\varphi)=g^{ij}(\nabla d\varphi)^\gamma_{ij}\frac{\partial}{\partial y^\gamma}=0. \label{eq2}
\end{equation}
\begin{defn}[\protect\cite{chen2}]
Let $(M,g)$ be a pseudo-Riemannian manifold of dimension $m$ and $\widehat{g}$ be another pseudo-Riemannian metric on $M$. We say that $\widehat{g}$ is a harmonic metric with respect to $g$ if and only if the identity map $I:(M,g)\rightarrow (M,\widehat{g})$ is harmonic.\end{defn}
By taking into account \eqref{eq1} and \eqref{eq2}, $\widehat{g}$ is harmonic with respect to $g$ if and only if
\begin{equation*}
g^{ij}(\widehat{\Gamma}^k_{ij}-\Gamma^k_{ij})=0,
\end{equation*}
or, equivalently
\begin{equation}
tr (G^{-1}(\widehat{\Gamma}^k-\Gamma^k))=0,\label{eq4}
\end{equation}
where $G$ is the matrix presentation of $g$ and $\Gamma^k,\widehat{\Gamma}^k$ are the matrices of the Christoffel symbols of $g$ and $\widehat{g}$ respectively.
\subsection{Geometry of tangent bundle}
Let $(M,g)$ be an $m-$dimensional pseudo-Riemannian manifold and $(TM,\pi, M)$ be its tangent bundle. We recall some facts on the geometry of the tangent bundle, following \cite{yano}. A system of local coordinates $(U,x^i), i=1,\ldots,m $, in $M$ induces on $TM$ a system of local coordinates $(\pi^{-1}(U),x^i,x^{\overline{i}}=u^i),$ $\overline{i}=m+i=m+1,\ldots,2m,$ where $(x^{\overline{i}}=u^i)$ are the components of vector fields on $M$ with respect to the coordinate vector fields $\{\partial_{x^1},\ldots,\partial_{x^m}\}$, here $\partial_{x^i}=\frac{\partial}{\partial{x^i}}.$

Denote the Levi-Civita connection of $g$ by $\nabla$. The tangent space $T_{(x,u)}TM$ at a point $(x,u)$ in $TM$ is the direct sum of the vertical subspace
$\mathcal{V}_{(x,u)}=\ker(d\pi\mid_{(x,u)})$ and the horizontal subspace $\mathcal{H}_{(x,u)}$, with respect to $\nabla$:
\begin{equation*}
T_{(x,u)}TM=\mathcal{H}_{(x,u)}\oplus \mathcal{V}_{(x,u)}.
\end{equation*}
We define the local vector fields on $TM$ as
\begin{equation*}
\delta_{x^i}=\partial _{x^i}-\Gamma^k_{0i}\partial_{x^{\overline{i}}},\quad \text{with}\quad \Gamma^k_{0i}=u^{h} \Gamma^{k}_{hi}
\end{equation*}
The local frame field $\{\delta_{x^i},\partial_{x^{\overline{i}}}\}_{1\leq i\leq m}$ is called the adapted local frame field on $M$, we easily see that the set $\{dx^i,\partial_{x^{\overline{i}}}^{\ast}\}_{1\leq i\leq m}$ is the coframe dual to the adapted local frame field $\{\delta_{x^i},\partial_{x^{\overline{i}}}\}_{1\leq i\leq m}$, where  $\partial_{x^{\overline{i}}}^{\ast}= dx^{\overline{i}}+x^{\overline{h}}\Gamma^i_{hj}dx^{j}$. Let $X|_{U} = X^{i}\partial_{x^i}$ be a local vector field on $M$. The vertical and the horizontal lifts of $X$ are defined with respect to the adapted frame field by:
\begin{equation*}
X^V=X^i\partial_{x^{\overline{i}}},\qquad X^H=X^i\delta _{x^i}.
\end{equation*}
The Sasaki lift metric $g^S$ and the horizontal lift metric $g^H$ of the metric $g$ on $TM$ are defined by
\begin{equation*}
\left\{
  \begin{array}{c}
   g^S(X^H,Y^H)= g(X,Y)\circ \pi,\\
   g^S(X^V,Y^H)= 0, \\
    g^S(X^V,Y^V)= g(X,Y)\circ \pi,
  \end{array}
\right. \qquad \left\{ \begin{array}{c}
   g^H(X^H,Y^H)= g^H(X^V,Y^V)=0,\\
      g^H(X^V,Y^H)= g(X,Y)\circ \pi,
  \end{array}
\right.
\end{equation*}for all vector fields $X,Y\in \Gamma(TM)$. The matrix representations of $g^S$ and $g^H$, respectively are given by
\begin{equation}
G^S=\left(
      \begin{array}{cc}
        g_{ij} & 0 \\
        0 &g_{ij} \\
      \end{array}
    \right), \qquad G^H=\left(
      \begin{array}{cc}
        0& g_{ij}  \\
        g_{ij}& 0 \\
      \end{array}
    \right).\label{eq5}
\end{equation}
Also the inverses of the matrices of $g^S$ and $g^H$ are given by
\begin{equation}
(G^S)^{-1}=\left(
      \begin{array}{cc}
        g^{ij} & 0 \\
        0 &g^{ij} \\
      \end{array}
    \right), \qquad (G^H)^{-1}=\left(
      \begin{array}{cc}
        0& g^{ij}  \\
        g^{ij}& 0 \\
      \end{array}
    \right).\label{eq6}
\end{equation}
The matrices of the Christoffel symbols of $g^S$ and $g^H$ are determined respectively as follows
\begin{align}
^S\Gamma^k&=\left(
             \begin{array}{cc}
               \Gamma^k_{ij} & \frac{1}{2}R^k_{hji}u^h \\
               \\
               \frac{1}{2}R^k_{hij}u^h & 0 \\
             \end{array}
           \right),\quad^S\Gamma^{\overline{k}}=\left(
             \begin{array}{cc}
              -\frac{1}{2}R^k_{ijh}u^h & \Gamma^k_{ij}  \\
              \\
                \Gamma^k_{ij}&0
             \end{array}
           \right),\label{eq7}\\
^H\Gamma^k&=\left(
             \begin{array}{cc}
              \Gamma^k_{ij}& \Gamma^k_{ij} \\
                \Gamma^k_{ij}&0 \\
             \end{array}
           \right),\qquad^H\Gamma^{\overline{k}}=0,\label{eq8}
 \end{align}
 where $R$ is the curvature tensor of $(M,g)$ and $R(\partial_{x_i},\partial_{x_j})\partial_{x_h}=R^k_{ijh}\partial_{x_k}$.
 Now, let $\omega$ be a $1-$form on $M$. The evaluation map is a function $i\omega:TM\rightarrow\mathbb{R}$, defined by $i\omega(p,u)=\omega_p(u)$. As a special case, for each function $f$, $f^C=i(df)$ is called the complete lift of the function $f$. These functions on $TM$ characterize vector fields on $TM$, since, for vector fields $\widetilde{X}$ and $\widetilde{Y}$ on $TM$, $\widetilde{X} (f^C)=\widetilde{Y} (f^C)$ if and only if $\widetilde{X}=\widetilde{Y}$, for all functions $f$. For each vector field $X$ on $M$, its complete lift $X^C$ is the vector field on $TM$ defined by
 \begin{equation*}
 X^C(f^C)=(Xf)^C.
 \end{equation*}
 For a pseudo-Riemannian manifold $(M,g)$, the tangent bundle $TM$ is naturally equipped with the complete lift metric
 \begin{equation*}
  g^C(X^C,Y^C)=(g(X,Y))^C,
 \end{equation*}
 for any $X,Y\in \Gamma(TM)$. The matrices presentations of $g^C$ and its inverse are given by (see \cite{calvino-louzao}, \cite{yano})
 \begin{equation}
 G^C=\left(
       \begin{array}{cc}
         x^{\overline{k}}\cfrac{\partial g_{ij}}{\partial x^k} & g_{ij} \\
         \\
         g_{ij} & 0
       \end{array}
     \right),\qquad
  (G^C)^{-1}=\left(
       \begin{array}{cc}
         0 & g^{ij} \\
         \\
         g^{ij} & x^{\overline{k}}\cfrac{\partial g^{ij}}{\partial x^k}
       \end{array}
     \right).\label{eq10}
 \end{equation}
 Note that $g^C$ is a pseudo-Riemannian metric of neutral signature. The Levi-Civita connections of $g^C$ and $g$ are related as follows
  \begin{equation}
   ^C\Gamma^k=\left(
  \begin{array}{cc}
  \Gamma^k_{ij} & 0 \\
  0 & 0
  \end{array}
  \right),\qquad  ^C\Gamma^{\overline{k}}=\left(
  \begin{array}{cc}
  x^{\overline{l}}\cfrac{\partial}{\partial x^l}\Gamma^k_{ij} & \Gamma^k_{ij} \\
  \\
  \Gamma^k_{ij} & 0            \end{array}          \right).
  \label{eq11}
 \end{equation}
  \section{Harmonic metrics lifted to the tangent bundle}
  \begin{thm}\label{thm1}
  Let $(M^m,g)$ be a pseudo-Riemannian manifold and $\widehat{g}$ is a pseudo-Riemannian metric on $M.$ The following statements are equivalent:

  \begin{itemize}
    \item[(i)] $\widehat{g}$ is harmonic with respect to $g$.
    \item[(ii)] $\widehat{g}^S$ is harmonic with respect to $g^S$.
    \item[(iii)] $\widehat{g}^H$ is harmonic with respect to $g^H$.
    \item[(iv)] $\widehat{g}^C$ is harmonic with respect to $g^C$.
  \end{itemize}
  \end{thm}
  \begin{proof}The equivalence between the items (i) and (ii) arises as follows:\\
  For the identity map $I:(TM,g^S)\rightarrow(TM,\widehat{g}^S)$ the harmonicity condition \eqref{eq4} can be expressed by
  \begin{equation}
  \left\{
    \begin{array}{c}
      tr((G^S)^{-1} (^S\widehat{\Gamma }^k- {}^S\Gamma ^k))=0,\\
      tr((G^S)^{-1} (^S\widehat{\Gamma }^{\overline{k}}-{}^S\Gamma ^{\overline{k}}))=0,\label{eq12} \\
    \end{array}
  \right.
  \end{equation}
  for all indices $k=1,\ldots,m$, where $^S\Gamma ^k,{}^S\Gamma ^{\overline{k}}$ are the matrices given by \eqref{eq7} and $^S\widehat{\Gamma}^k,{}^S\widehat{\Gamma }^{\overline{k}}$ the similar matrices for $^S\widehat{g}$. By using \eqref{eq6} and \eqref{eq7}, the first equation of \eqref{eq12} reduces to
  \begin{equation}
  tr(G^{-1}(\widehat{\Gamma }^k-\Gamma^k))=0,\qquad \forall k=1,\ldots,m.\label{eq13}
  \end{equation}
  Now, denoting the curvature tensor of $g$ and $ \widehat{g}$ by $R$ and $\widehat{R}$, respectively, and taking into account \eqref{eq6} and \eqref{eq7}, it follows that the second equation in \eqref{eq12} yields
  \begin{equation}tr(g^{ij}(\widehat{R}^k_{ij0}-R^k_{ij0}))=0,\label{eq14}
  \end{equation}
 where $R^k_{ij0}=R^k_{ijh}u^h$ and $\widehat{R}^k_{ij0}=\widehat{R}^k_{ijh}u^h$. Equivalently, \eqref{eq14} can be
 written as follows
 \begin{equation}
  tr(G^{-1}(\widehat{R}^k-R^k))=0,\qquad \forall k=1,\ldots,m,\label{eq14-1}
  \end{equation}
 where $R^k$ and $\widehat{R}^k$ are the matrices given by $R^k=(R^k_{ij0})$, $\widehat{R}^k=(\widehat{R }^k_{ij0})$, $i,j=1,\ldots,m.$ Since $G^{-1}$ is symmetric and $(\widehat{R}^k-R^k)$ is antisymmetric, then by a result of linear algebra \eqref{eq14-1} is always satisfied. The relations \eqref{eq13} expresses the harmonicity condition of the identity map $I:(M,g)\rightarrow (M,\widehat{g})$. Hence, the metric $\widehat{g}$ is harmonic with respect to $g$, which yields the equivalence between (i) and (ii).

  Using \eqref{eq4}, the harmonicity condition of $\widehat{g}^H$ with respect to $g^H$ can be expressed by:
  \begin{equation}
  \left\{
    \begin{array}{c}
      tr((G^S)^{-1} (^H\widehat{\Gamma }^k-{}^H\Gamma ^k))=0,\\
      tr((G^S)^{-1} (^H\widehat{\Gamma }^{\overline{k}}-{}^H\Gamma ^{\overline{k}}))=0,\label{eq15} \\
    \end{array}
  \right.
  \end{equation}
 for all indices $k=1,\ldots,m$, where $^H\Gamma ^k,{}^H\Gamma ^{\overline{k}}$ are the matrices given by \eqref{eq8} and $^H\widehat{\Gamma}^k,{}^H\widehat{\Gamma }^{\overline{k}}$ the similar matrices for $^H\widehat{g}$. So by \eqref{eq8} the system \eqref{eq15} gives
 \begin{equation*}
  tr(G^{-1}(\widehat{\Gamma }^k-\Gamma^k))=0\qquad \forall k=1,\ldots,m.
  \end{equation*}
  This equation shows exactly the same conditions of harmonicity of the identity map $I:(M,g)\rightarrow(M,\widehat{g})$.

  The third equivalence between (i) and (iv) follows by the same method as above by using \eqref{eq10} and \eqref{eq11}. Thus the proof is complete.
  \end{proof}
  \begin{rem}Theorem \ref{thm1}, recovers the results due to \cite{bejan,zaeim1} and \cite{zaeim2}.
  \end{rem}
  \section{Harmonic metrics on Egorov spaces}
The Egorov spaces are Lorentzian manifolds $(\mathbb{R}^m,g_f)$, $m\geq 3$, where $f$ is a positive smooth function of a real variable and
\begin{equation}
g_f=f(x^m)\sum^{m-2}_{i=1}(dx^i)^2+2dx^{m-1}dx^m.\label{eq16}
\end{equation}
 These manifolds are named after I. P. Egorov, who first introduced and studied them in \cite{egorov}. In \cite{batat}, W. Batat, G. Calvaruso and B. De Leo  described some curvature properties of Egorov spaces. Now, we study the harmonicity condition for Lorentzian metrics which are raised from Egorov space.
 \begin{thm}\label{thm2}
 Let $(\mathbb{R}^m,g_f)$, $m\geq 3$ be an Egorov space with the metric defined by \eqref{eq16}. Then, the metric
 \begin{equation}\label{eq17}
\widehat{g}_{\widehat{f}}=\widehat{f}(x^m)\sum^{m-2}_{i=1}(dx^i)^2+2dx^{m-1}dx^m,
\end{equation}
is harmonic with respect to $g_f$ if and only if
\begin{equation*}
\widehat{f}'=f'.
\end{equation*}
 \end{thm}
 \begin{proof}Let $(\mathbb{R}^m,g_f)$, $m\geq 3$ be an Egorov space with the metric defined by \eqref{eq16}, and $\widehat{g}_{\widehat{f}}$ be the metric
 \begin{equation*}
\widehat{g}_{\widehat{f}}=\widehat{f}(x^m)\sum^{m-2}_{i=1}(dx^i)^2+2dx^{m-1}dx^m,
\end{equation*}
which is harmonic with respect to the metric $g_f$. Put $f(x^{m})=f$, the matrix presentation of the metric $g_f$ is
\begin{equation}\label{eq18}
G_f=\left(
\begin{array}{c|c}
\begin{array}{ccc}
  f &  &O\\
 & \ddots& \\
0 & & f\end{array}&O\\
\hline\\
O&
\begin{array}{cc}
 0&1 \\
1&0    \end{array}
\end{array}\right).
\end{equation}
It's obvious that
\begin{equation}\label{eq19}
(G_f)^{-1}=\left(
\begin{array}{c|c}
\begin{array}{ccc}
  \frac{1}{f} &  &O\\
 & \ddots& \\
O & & \frac{1}{f}\end{array}&O\\
\hline\\
O&
\begin{array}{cc}
 0&1 \\
1&0    \end{array}
\end{array}\right).
\end{equation}
Put $\Gamma^k=(\Gamma^k_{ij})$, $i,j=1,\ldots,m$ and $k=1,\ldots,m,$ then from \cite{batat}, the only possible non-vanishing Christoffel symbols are the following ones:
\begin{equation}\label{eq19-1}
    \Gamma^{m-1}_{ii}=\frac{-f'}{2},\;\;\;\Gamma^i_{im}=\frac{f'}{2f},\;\;\;i=1,\cdots,m-2.
\end{equation}
Thus
\begin{equation*}
\Gamma^1=\left( \begin{array}{c|c}
O&
\begin{array}{c}
\cfrac{f'}{2f}\\
0\\\vdots\\0\end{array}\\
\hline
\begin{array}{cccc}
\cfrac{f'}{2f}&
0&\cdots& 0\end{array}&0
\end{array}
           \right),
\end{equation*}
\begin{equation*}
\Gamma^2=\left( \begin{array}{c|c}
O&
\begin{array}{c}
0\\
\cfrac{f'}{2f}\\0\\\vdots\\0\end{array}\\
\hline
\begin{array}{ccccc}
0&
\cfrac{f'}{2f}&0&\cdots&0\end{array}&0
\end{array}
           \right).
\end{equation*}
And so on until we arrive to
\begin{equation*}
\Gamma^{m-2}=\left( \begin{array}{c|c}
O&
\begin{array}{c}
0\\
\vdots\\0\\\cfrac{f'}{2f}\\0\end{array}\\
\hline
\begin{array}{ccccc}
0&
\cdots&0&\cfrac{f'}{2f}&0\end{array}&0
\end{array}
           \right)
\end{equation*}
and
\begin{equation*}
\Gamma^{m-1}=\left(
\begin{array}{c|c}
\begin{array}{ccc}
-\cfrac{f'}{2} &  & O \\
& \ddots &  \\
O& & -\cfrac{f'}{2 }\end{array}&O\\
\hline\\
O& \begin{array}{cc}
0 & 0 \\
0 & 0
\end{array}
\end{array}
\right),\qquad\Gamma^{m}=0.
\end{equation*}
The Levi-Civita components $\widehat{\Gamma}^1,\ldots ,\widehat{\Gamma}^m$ of the Lorentzian metric given in \eqref{eq17}, are obviously deduced by replacing $\widehat{f}$ instead of $f$ in $\Gamma^1,\ldots ,\Gamma^m$. By direct computations, we find the only non-zero trace as follows:
\begin{equation*}
tr(G^{-1}_f(\widehat{\Gamma}^{m-1}-\Gamma^{m-1}))=\cfrac{(m-2)(f'-\widehat{f}')}{2f}.
\end{equation*}
Thus, by virtue of \eqref{eq4}, the identity map $I:(\mathbb{R}^m,g_f)\rightarrow (\mathbb{R}^m,\widehat{g}_{\widehat{f}})$ is harmonic if and only if $\widehat{f}'=f'$, which completes the proof.
 \end{proof}
 \begin{ex}\label{ex1}
 Let $(M,g)$ be an $m-$dimensional pseudo-Riemannian manifold and $TM$ be its tangent bundle. The Sasaki and horizontal lift metrics of $g$ with respect to the local coordinates $(x^i,x^{\overline{i}})$ of $TM$ are respectively \cite{yano}
 \begin{align}
 g^S&=g_{ij}dx^{i} dx^{j}+g_{ij}\partial_{x^{\overline{i}}}^{\ast}\partial_{x^{\overline{j}}}^{\ast},\label{eqS1}\\
  g^H&=2g_{ij}dx^{i}\partial_{x^{\overline{j}}}^{\ast}\label{eqS2}.
 \end{align}
Then, by virtue of \eqref{eq10}, \eqref{eq16}, \eqref{eq19-1}, \eqref{eqS1} and \eqref{eqS2}, for an  Egorov space $(\mathbb{R}^m,g_f)$, $m\geq3$, the Sasaki, horizontal and complete lift metrics of $g_f$ with respect to the local coordinates $(x^i,x^{\overline{i}})$ of the tangent bundle of $\mathbb{R}^m$ are respectively
\begin{align*}
 g^S_f&=f\sum^{m-2}_{i=1}(dx^i)^2+2dx^{m-1}dx^m+f\sum^{m-2}_{i=1}(dx^{\overline{i}})^2+
 2dx^{\overline{m-1}}dx^{\overline{m}}\\&+ f'\sum^{m-2}_{i=1}x^{\overline{i}}dx^mdx^{\overline{i}}-
 f'\sum^{m-2}_{i=1}x^{\overline{i}}dx^{\overline{m}}dx^{i}+ \frac{(f')^{2}}{4f}\sum^{m-2}_{i=1}(x^{\overline{i}})^{2}(dx^m)^{2},\\
  g^H_f&=2f(x^m)\sum^{m-2}_{i=1}dx^idx^{\overline{i}}+2dx^{m-1}dx^{\overline{m}}+2dx^{\overline{m-1}}dx^m,\\
  g^C_f&=2f(x^m)\sum^{m-2}_{i=1}dx^idx^{\overline{i}}+2dx^{m-1}dx^{\overline{m}}+
  2dx^{\overline{m-1}}dx^m+x^{\overline{m}}f'(x^m)\sum^{m-2}_{i=1}(dx^i)^2.
 \end{align*}
  By using Theorem \ref{thm1} and Theorem \ref{thm2}, the metric:
  \begin{enumerate}
   \item $g_{e^{x^m}+K}$ is harmonic with respect to $g_{e^{x^m}}$,
   \item $g^S_{e^{x^m}+K}$ is harmonic with respect to $g^S_{e^{x^m}}$,
   \item $g^H_{e^{x^m}+K}$ is harmonic with respect to $g^H_{e^{x^m}}$,
   \item $g^C_{e^{x^m}+K}$ is harmonic with respect to $g^C_{e^{x^m}}$, where $K$ is positive constant.
   \end{enumerate}
\end{ex}
 \section{Harmonic metrics lifted to the cotangent bundle}
 Let $(M,g)$ be an $m-$dimensional pseudo-Riemannian manifold, $T^*M$ its cotangent bundle, and $\widetilde{\pi}$ the natural projection $T^*M\rightarrow M$. We recall some basic properties on the geometry of the cotangent bunde following \cite{salimov,yano}. A system of local coordinates $(U,x^i), i=1,\ldots,m$ in $M$ induces on $T^*M$ a system of local coordinates $( \widetilde{\pi}^{-1}(U),x^i,x^{\widetilde{i}}=p_i),$ $\widetilde{i}=m+i=m+1,\ldots,2m$, where $(x^{\widetilde{i}}=p_i)$ is the components of covector $p$ in each cotangent space $T^*_xM$, $x\in U$ with respect to the natural coframe $\{dx^i\}$.

 Let $X=X^i\frac{\partial}{\partial x^i}=X^i\partial_{x^i}$ and $\omega=\omega_idx^i$ be the local expressions in $U\subset M$ of a vector and covector (1-form) fields $X$ and $\omega$, respectively.
 Then the horizontal lift $^H\!X$ of $X$ and the vertical lift $^V\!\omega$ of $\omega$ are given respectively, by
 \begin{align*}
 ^H\!X&=X^i\partial_{x^i}+p_h\Gamma^h_{ij}X^j\partial_{x^{\widetilde{i}}};\\
  ^V\!\omega&=\omega_i\partial_{x^{\widetilde{i}}}
   \end{align*}
   with respect to the natural frame $\{\partial_{x^i},\partial_{x^{\widetilde{i}}}\}$, where $\Gamma^h_{ij}$ are components of the Levi-Civita connection $\nabla$ of $M$. Define the local vector fields on $T^*M$ as
 \begin{equation*}
    \widetilde{ \delta}_{x^i}=\partial_{x^i}+p_a\Gamma^a_{hi}\partial_{x^{\widetilde{i}}},
 \end{equation*}
 we call the set $\{\widetilde{ \delta}_{x^i},\partial_{x^{\widetilde{i}}}\}$, the frame adapted to $\nabla$, it's clear that the set $\{dx^i,\partial_{x^{\widetilde{i}}}^{\ast}\}_{1\leq i\leq m}$ is the coframe dual to frame $\{\widetilde{ \delta}_{x^i},\partial_{x^{\widetilde{i}}}\}$, where $\partial_{x^{\widetilde{i}}}^{\ast}= dx^{\widetilde{i}}-p_a\Gamma^a_{hi}dx^{h}$. For each $x\in M$ the scalar product $g^{-1}=(g^{ij})$ is defined on the cotangent space $\widetilde{\pi}^{-1}(x)=T^*_xM$ by
 \begin{equation*}
    g^{-1}(\omega,\theta)=\omega_i\theta_jg^{ij},
 \end{equation*}
for all covectors $\omega$ and $\theta$ on $M$.

A Sasakian metric $g^{\widetilde{S}}$ is defined on $T^*M$ by the three equations
  \begin{equation*} \left\{
     \begin{array}{c}
       g^{\widetilde{S}}(^V\!\omega,^V\!\theta )={}^V\!(g^{-1}(\omega,\theta))=g^{-1}(\omega,\theta)\circ \pi, \label{eq22}\\
       g^{\widetilde{S}}(^V\!\omega,^H\!Y ) =0,  \\
       g^{\widetilde{S}}(^H\!X,^H\! Y ) ={}^V\!(g(X,Y))=g(X,Y)\circ \pi, \\
     \end{array}
   \right.
   \end{equation*}
   for any $X,Y\in \Gamma(TM)$ and any covectors $\omega$ and $\theta$ on $M$. The matrix representations of $g^{\widetilde{S}}$ and its inverse are given by, respectively
   \begin{equation}
   G^{\widetilde{S}}=\left(
         \begin{array}{cc}
           g_{ij}& 0 \\
           0 & g^{ij} \\
         \end{array}
       \right),\qquad (G^{\widetilde{S}})^{-1}=\left(
         \begin{array}{cc}
           g^{ij}& 0 \\
           0 & g_{ij} \\
         \end{array}
       \right),\label{eq25}
   \end{equation}
   with respect to the adapted frame $\{\widetilde{ \delta}_{x^i},\partial_{x^{\widetilde{i}}}\}$. The matrices of the Christoffel symbol of $g^{\widetilde{S}}$ are given as follows :
 \begin{align}
^{\widetilde{S}}\Gamma^k&=\left(
                  \begin{array}{cc}
                    \Gamma^k_{ij} & \frac{1}{2}p_m(R^k_{.i.})^{jm} \\
                    \\
                    \frac{1}{2}p_m(R^k_{.j.})^{im}& 0 \\
                  \end{array}
                \right),\label{eq26}\\
 ^{\widetilde{S}}\Gamma^{\widetilde{k}}&=\left(
                  \begin{array}{cc}
                    \frac{1}{2}p_m R^m_{ijk}& -\Gamma^j_{ik}  \\
                   \\
                    -\Gamma^i_{jk}& 0 \\
                  \end{array}
                \right),\label{eq27}
  \end{align}
where $(R^k_{.i.})^{jm}=g^{kt}g^{js}R^m_{tis}.$
\begin{thm}\label{thm3}
Let $(M,g)$ be a pseudo-Riemannian manifold and $\widehat{g}$ is a pseudo-Riemannian metric on $M$. The following statements are equivalent:
\begin{itemize}
    \item[(i)] $\widehat{g}$ is harmonic with respect to $g$.
    \item[(ii)] $\widehat{g}^{\widetilde{S}}$ is harmonic with respect to $g^{\widetilde{S}}$.
\end{itemize}
\end{thm}
\begin{proof}For the identity map $I:(T^*M,g^{\widetilde{S}})\rightarrow(T^*M,\widehat{g}^{\widetilde{S}})$ the harmonicity condition \eqref{eq4} can be described by:
\begin{equation}
\left\{
  \begin{array}{c}
    tr((G^{\widetilde{S}})^{-1}(^{\widetilde{S}}\widehat{\Gamma}^k- {}^{\widetilde{S}}\Gamma^k))=0,\\
    tr((G^{\widetilde{S}})^{-1}(^{\widetilde{S}}\widehat{\Gamma}^{\widetilde{k}}- {}^{\widetilde{S}}\Gamma^{\widetilde{k}}))=0, \\
  \end{array}
\right.\label{eq28}
\end{equation}
for all indices $k=1,\ldots,m,$ by virtue of \eqref{eq25} and \eqref{eq26}, the first equation of \eqref{eq28} reduces to
\begin{equation*} tr(G^{-1}(\widehat{\Gamma}^k-\Gamma^k))=0\qquad \forall k=1,\ldots,m.
\end{equation*}Now, taking into account \eqref{eq25} and \eqref{eq27}, it follows that the second equation in \eqref{eq28} yields
\begin{equation} tr(g^{ij}(\widehat{R}^o_{ijk}-R^o_{ijk}))=0, \label{eq29}
\end{equation}
where $R^o_{ijk}=p_h R^h_{ijk}$ and $\widehat{R}^o_{ijk}=p_h \widehat{R}^h_{ijk}$. Equivalently, \eqref{eq29} can be
 written as follows
 \begin{equation}
  tr(G^{-1}(\widehat{R}_k-R_k))=0,\qquad \forall k=1,\ldots,m,\label{eq29-1}
  \end{equation}
 where $R_k$ and $\widehat{R}_k$ are the matrices given by $R_k=(R^o_{ijk})$, $\widehat{R}_k=(\widehat{R}^o_{ijk})$ , $i,j=1,\ldots,m.$
Since $G^{-1}$ is symmetric and $(\widehat{R}_k-R_k))$ is antisymmetric, then \eqref{eq29-1} is always verified. Hence the metric $\widehat{g}$ is harmonic with respect to $g$, which gives the equivalence between (i) and (ii).
\end{proof}
From Theorem \ref{thm1} and Theorem \ref{thm3}, we obtain
\begin{thm}\label{thm4}
  Let $(M^m,g)$ be a pseudo-Riemannian manifold and $\widehat{g}$ is a pseudo-Riemannian metric on $M.$ The following statements are equivalent:
  \begin{itemize}
    \item[(i)] $\widehat{g}$ is harmonic with respect to $g$.
    \item[(ii)] $\widehat{g}^S$ is harmonic with respect to $g^S$.
    \item[(iii)] $\widehat{g}^H$ is harmonic with respect to $g^H$.
    \item[(iv)] $\widehat{g}^C$ is harmonic with respect to $g^C$.
    \item[(v)] $\widehat{g}^{\widetilde{S}}$ is harmonic with respect to $g^{\widetilde{S}}$.
  \end{itemize}
  \end{thm}
\section{Applications}
In this section, we apply Theorem \ref{thm3} on four-dimensional Walker manifolds of signature (2,2), on G\"{o}del-type spacetimes, on non-reductive homogeneous manifolds of dimension four and on Egorov spaces.
\subsection{Four-dimensional Walker manifolds of signature (2,2)}
Let $(M,g)$ be a four-dimensional Walker manifold of signature (2,2). We know from \cite{brozos} that there exists local coordinates $(x^1,x^2,x^3,x^4)$ such that $g$ is described as
\begin{equation}
g_{a,b,c}=2dx^1dx^4+2dx^2dx^3+a(dx^3)^2+b(dx^4)^2+2cdx^3dx^4\label{equ1}
\end{equation}
where $a,b,c$ are arbitrary smooth functions on $M$. Let $\widehat{g}_{\widehat{a},\widehat{b},\widehat{c}}$ an arbitrary Walker metric, the authors in \cite{bejan} proved that $\widehat{g}_{\widehat{a},\widehat{b},\widehat{c}}$ is harmonic with respect to $g_{a,b,c}$ if and only if the following relations are established
\begin{equation}
\left\{
  \begin{array}{c}
   a_t+c_x= \widehat{a}_t+\widehat{c}_x\\
    c_t+b_x=\widehat{c}_t+\widehat{b}_x \\
  \end{array}
\right. .\label{equ2}
\end{equation}
By virtue of Theorem \ref{thm3} and \eqref{equ2}, we get the following
\begin{thm} Let $(M,g_{a,b,c})$ be a four-dimensional Walker manifold and $\widehat{g}_{\widehat{a},\widehat{b},\widehat{c}}$ an arbitrary Walker metric on $M$. Then the Sasaki lift metric $\widehat{g}_{\widehat{a},\widehat{b},\widehat{c}}^{\widetilde{S}}$ is harmonic with respect to
$ g_{a,b,c}^{\widetilde{S}}$ if and only if \eqref{equ2} holds.
\end{thm}
\subsection{G\"{o}del-type spacetimes}
Let $(M,g)$ be a G\"{o}del-type spacetime where $g$ is described by the following Lorentzian metric
\begin{equation*}
g=[dt + H(r)d\varphi]^2-dr^2-P^2(r)d\varphi^2-dz^2,
\end{equation*}
here $t$ is the time variable and $(r,\varphi,z)$ are the usual cylindrical coordinates and so $r\geq0,$ $\varphi\in\mathbb{R}$
(undetermined for $r=0)$ and $z\in \mathbb{R}$. Moreover, $H(r)$ and $P(r)$ are arbitrary smooth functions on $M$ and $g$ is non degenerate where $P(r)\neq0$. We can rewrite $g$ as following
\begin{equation}
g=[dx^1 + H(x^2)dx^3]^2-(dx^2)^2-P^2(x^2)(dx^3)^2-(dx^4)^2, \label{equ3}
\end{equation}
such that $(x^1,x^2,x^3,x^4)=(t,r,\varphi,z)$. The authors in \cite{zaeim2} proved that the G\"{o}del-type metric
\begin{equation}
\widehat{g}=[dx^1 + \widehat{H}(x^2)dx^3]^2-(dx^2)^2-\widehat{P}^2(x^2)(dx^3)^2-(dx^4)^2, \label{equ4}
\end{equation} is harmonic with respect to $g$ \eqref{equ3} if and only if
\begin{equation}
\widehat{H}'(x^2)(\widehat{H}(x^2)-H(x^2))-\widehat{P}(x^2)\widehat{P}'(x^2)+P(x^2)P'(x^2)=0. \label{equ5}
\end{equation}
Using Theorem \ref{thm3} and \eqref{equ5}, we deduce
\begin{thm} Let $(M,g)$ be a G\"{o}del-type spacetime with the metric \eqref{equ3} and $\widehat{g}$ the Lorentzian metric given by \eqref{equ4}. Then the Sasaki lift metric $\widehat{g}^{\widetilde{S}}$ is harmonic with respect to $g^{\widetilde{S}}$ if and only if \eqref{equ5} is satisfied.
\end{thm}
\subsection{Non-reductive homogeneous manifolds of dimension four}
A (connected) pseudo-Riemannian manifold $(M,g)$ is said to be homogeneous if it admits a group $G$ of isometries, acting transitively on it. In this case, $(M,g)$ can be identified with $(G/H,g)$, where $H$ is the isotropy group at a fixed point of $M$  and $g$  is invariant pseudo-Riemannian metric. A homogeneous pseudo-Riemannian manifold $(M,g)$ is reductive if the Lie algebra $\mathfrak{g}$ of $G$ may be decomposed into a vector space direct sum
$\mathfrak{g}=\mathfrak{h}\oplus\mathfrak{m}$ where $\mathfrak{m}$ is an $Ad(H)$-invariant complement to $\mathfrak{h}$. Non-reductive homogeneous manifolds of dimension four were classified in \cite{fels}, in terms of the corresponding non-reductive Lie algebra presentations. This classification contains 8 classes $A_1,A_2,A_3,A_4,A_5$ and $B_1,B_2,B_3$. The classes $A_1,$ $A_2,$ $A_3$ contain both Lorentzian and neutral examples, classes $A_4,A_5$ are just Lorentzian and $B_1,B_2,B_3$ are always of neutral signature. Also the description in coordinates system for the invariant metrics on the spaces mentioned were obtained in \cite{calvaruso}. Any invariant metric on a non-reductive homogeneous four manifold is described by some arbitrary coefficients $a,b,\ldots$, see \cite{calvaruso,zaeim1}. We take an arbitrary invariant metric $\widehat{g}$, which is defined by arbitrary coefficients $\widehat{a},\widehat{b},\ldots$. The authors in \cite{zaeim1} obtained the conditions under which the metric $\widehat{g}$ is harmonic with respect to the metric $g$. By using Theorem \ref{thm3} and Theorem 5.1 in \cite{zaeim1}, we obtain the following
\begin{thm} Let $(M=G/H,g)$ be a non-reductive homogeneous manifold of dimension four, equipped with invariant metric $g$. Then the Sasaki metric $\widehat{g}^{\widetilde{S}}$ is harmonic with respect to $g^{\widetilde{S}}$ if and only if one of the following cases occurs
\begin{enumerate}
  \item $(G/H,g)$ is of type $A_1$ and $c\widehat{a}=a\widehat{c}$.
  \item $(G/H,g)$ is of type $A_2$ and one of the following cases occurs
  \begin{enumerate}
    \item $\alpha=0,$ $c\widehat{a}=a\widehat{c},$
    \item $\alpha=\cfrac{1}{4},$ $a\widehat{d}=d\widehat{a},$
    \item $\alpha= arbitrary,$ $c\widehat{a}=a\widehat{c},$ $a\widehat{d}=d\widehat{a}.$
  \end{enumerate}
  \item $(G/H,g)$ is of type $A_3$ for $\varepsilon=\pm 1$ and $b\widehat{a}=a\widehat{b},$ $b\widehat{c}=c\widehat{b}$.
  \item $(G/H,g)$ is of type $A_4,A_5,B_2$ or $B_3$.
  \item $(G/H,g)$ is of type $B_1$ and $c\widehat{a}=a\widehat{c},$ $a\widehat{d}=d\widehat{a}$.
\end{enumerate}

\end{thm}
 \subsection{Egorov spaces}
 Let $(\mathbb{R}^m,g_f)$, $m\geq 3$ be an Egorov space. The Sasaki lift metric of $g_f$ with respect to the coordinates $(x^i,x^{\widetilde{i}})$ in the cotangent bundle $T^{\ast}\mathbb{R}^m$ is given by
\begin{equation}\label{eqF}
   g_f^{\widetilde{S}}=(g_{f})_{ij}dx^{i}dx^{j}+
   (g_{f})^{ij}\partial_{x^{\widetilde{i}}}^{\ast}\partial_{x^{\widetilde{j}}}^{\ast},
\end{equation}
Then, taking into account \eqref{eq16}, \eqref{eq19-1} and \eqref{eqF}, we get 
\begin{align*}
g_f^{\widetilde{S}}&=f\sum^{m-2}_{i=1}(dx^i)^2+2dx^{m-1}dx^m+\frac{1}{f}\sum^{m-2}_{i=1}(dx^{\widetilde{i}})^2+
 2dx^{\widetilde{m-1}}dx^{\widetilde{m}},\\&+ \frac{f'}{f}\sum^{m-2}_{i=1}x^{\widetilde{m-1}}dx^{\widetilde{i}}dx^{i}-
 \frac{f'}{f}\sum^{m-2}_{i=1}x^{\widetilde{i}}dx^{\widetilde{m-1}}dx^{i}+ \frac{(f')^{2}}{4f}(x^{\widetilde{m-1}})^{2}\sum^{m-2}_{i=1}(dx^i)^{2}.
 \end{align*}
Combining Theorem \ref{thm2} and Theorem \ref{thm3}, we deduce
 \begin{thm} Let $(\mathbb{R}^m,g_f)$, $m\geq 3$ be an Egorov space. Then the Sasaki metric $\widehat{g}_{\widehat{f}}^{\widetilde{S}}$ is harmonic with respect to $g_{f}^{\widetilde{S}}$ if and only if
\begin{equation*}
\widehat{f}'=f'.
\end{equation*}
 \end{thm}


\begin{thebibliography}{99}
\bibitem{baird} Baird, P., Wood, J. C.: Harmonic morphisms between Riemannain manifolds, Clarendon Press Oxford (2003)

\bibitem{batat} Batat, W., Calvaruso, G., De Leo, B.: Curvature properties of Lorentzian manifolds with large isometry groups. Math. Phys. Anal. Geom., \textbf{12}, 201--217 (2009)

\bibitem{bejan} Bejan, C. L., Dru\c{t}\u{a}-Romaniuc, S. L.: Structures which are harmonic with respect
to Walker metrics. Mediterr. J. Math., \textbf{12}, 481--496 (2015)

\bibitem{brozos} Brozos-V\'{a}zquez, M., Garc\'{\i}a-R\'{\i}o, E., Gilkey, P., Nik\v{c}ev\'{\i}, S., V\'{a}zquez-Lorenzo, R. : The Geometry of Walker Manifolds. Synthesis Lectures on Mathematics and Statistics 5, Morgan and Claypool Publishers, St. Louis (2009)

\bibitem{calvaruso} Calvaruso, G., Fino, A., Zaeim, A.: Homogeneous geodesics of non-reductive homogeneous pseudo-Riemannian 4-manifolds. Bull. Braz. Math. Soc. (N.S.), \textbf{46}, 23--64 (2015)

\bibitem{calvino-louzao} Calvi\~{n}o-Louzao, E., Garc\'{\i}a-R\'{\i}o, E., Sixto-Neira, M., V\'{a}zquez-Abal, M. E.: Biharmonic maps on tangent and cotangent bundles. J. Geom. Phys., \textbf{101}, 1--10 (2016)

\bibitem{chen1} Chen, B. Y.: Differential geometry of identity maps: A survey. Mathematics, \textbf{36}(8), 1--33 (2020)

\bibitem{chen2} Chen, B. Y., Nagano, T.: Harmonic metrics, harmonic tensors and Gauss maps. J. Math.
Soc. Japan, \textbf{36}, 295--313 (1981)

\bibitem{eells1} Eells, J., Lemaire, L.: A report on harmonic maps. Bull. London Math. Soc., \textbf{10}, 1--68 (1978)

\bibitem{eells2} Eells, J., Lemaire, L.: Another report on harmonic maps. Bull. London Math. Soc., \textbf{20}, 385--524 (1988)

\bibitem{eells3} Eells, J., Sampson, J. H.: Harmonic mappings of Riemannian manifolds. Am. J. Math., \textbf{86}, 109--160 (1964)

\bibitem{egorov} Egorov, I. P.: Riemannian spaces of the first three lacunary types in the geometric sense (Russian). Dokl. Akad. Nauk. SSSR \textbf{150}, 730--732 (1963)

\bibitem{fels} Fels, M. E., Renner, A. G.: Non-reductive homogeneous pseudo-Riemannian manifolds of dimension four. Can. J. Math., \textbf{58}, 282--311 (2006)

\bibitem{gauduchon} Gauduchon, P.: Harmonic Mappings, Twistors and Sigma Models. World Scientific Publishing Company, (1988)

\bibitem{salimov} Salimov, A. A., Agca, F.: Some properties of Sasakian metrics in cotangent bundles. Mediterr. J. Math., \textbf{8}, 243--255 (2011)

\bibitem{xin} Xin, Y.: Geometry of Harmonic Maps. Birkh\"{a}user Boston, Boston, MA (1996)

\bibitem{yano} Yano, K., Ishihara, S.: Tangent and Cotangent Bundles.
Dekker, New York (1973)

\bibitem{zaeim1} Zaeim, A., Atashpeykar, P.: Harmonic metrics on four dimensional non-reductive
homogeneous manifolds. Czechoslovak Math. J. \textbf{68}(2), 475--490 (2018)

\bibitem{zaeim2} Zaeim, A., Jafari, M., Yaghoubi, M.: Harmonic metrics on G\"{o}del type spacetiemes. Int J.
Geom. Methods Mod. Phys. \textbf{17}(6), 2050092 (2020)

\end{thebibliography}
   \end{document}